\documentclass[12pt]{article}

\usepackage{amsfonts}
\usepackage{amsmath}
\usepackage{amssymb}
\usepackage{color}
\usepackage{amsthm}
\usepackage{hyperref}
\usepackage{makecell}
\usepackage{graphicx}
\usepackage{multirow}

\theoremstyle{plain}
\newtheorem{theorem}{Theorem}

\newtheorem{lemma}[theorem]{Lemma}

\theoremstyle{definition}
\newtheorem{example}[theorem]{Example}

\theoremstyle{remark}
\newtheorem{remark}[theorem]{Remark}

\title{Approximations for the number of maxima and near-maxima in independent data} \author{Fraser Daly\footnote{Department of Actuarial Mathematics and Statistics, and the Maxwell Institute for Mathematical Sciences, Heriot--Watt University, Edinburgh EH14 4AS, UK.  E-mail: F.Daly@hw.ac.uk}} \date{\today}

\begin{document}

\maketitle

\noindent{\bf Abstract} 
In the setting where we have $n$ independent observations of a random variable $X$, we derive explicit error bounds in total variation distance when approximating the number of observations equal to the maximum of the sample (in the case where $X$ is discrete) or the number of observations within a given distance of an order statistic of the sample (in the case where $X$ is absolutely continuous). The logarithmic and Poisson distributions are used as approximations in the discrete case, with proofs which include the development of Stein's method for a logarithmic target distribution. In the absolutely continuous case our approximations are by the negative binomial distribution, and are established by considering negative binomial approximation for mixed binomials. The cases where $X$ is geometric, Gumbel and uniform are used as illustrative examples.  
\vspace{12pt}


\vspace{12pt}


\section{Introduction and main results}\label{sec:intro}

Let $X,X_1,X_2,\ldots$ be independent and identically distributed (i$.$i$.$d$.$) random variables, and denote by $M_n=\max\{X_1,\ldots,X_n\}$ the maximum of $n$ of these random variables. In the discrete setting where the $X_i$ take positive integer values, our aim is to establish explicit error bounds in approximations for the number of these $X_i$ that are equal to their maximum. We denote this quantity by $K_n$. That is,
\begin{equation}\label{eq:kdef}
K_n=|\{i\in\{1,\ldots,n\}:X_i=M_n\}|\,.
\end{equation}
Eisenberg \cite{eisenberg09} discusses applications of this statistic to sporting competitions, modelling the following scenario: assuming a set of players of equal ability, how many are tied for the record performance? Similarly, $K_n$ finds applications in the reliability of systems composed of $n$ independent components with identically distributed lifetimes, and in the output of randomised selection algorithms (see \cite{bruss03} for a brief discussion of this final application).      

Here we are motivated by the work of Brands \emph{et al$.$} \cite{brands94}, who show that, while $K_n$ does not in general converge to a distributional limit as $n\to\infty$, there are nevertheless good approximations to the distribution of $K_n$. Motivated by a coin-tossing problem of R\"ade \cite{rade91}, Brands \emph{et al$.$} pay particular attention to the case where $X$ is geometrically distributed with parameter $p$, and observe that if $p$ is constant (i.e., independent of $n$) the distribution of $K_n$ is close to logarithmic, while for other choices of $p$ it may be close to Poisson. Our first principal aim here is to complement these observations with explicit error bounds in total variation distance for the approximation of the distribution of $K_n$ by either a logarithmic or Poisson distribution, and thus obtain simple approximations relevant to the applications discussed above together with quantitative error bounds. The number of maxima in the case where $X$ is geometrically distributed has also been paid particular attention by Kirschenhofer and Prodinger \cite{kirschenhofer96} and Olofsson \cite{olofsson99}, and we refer the interested reader to the work of Eisenberg \cite{eisenberg09} and Bruss and Gr\"ubel \cite{bruss03} and references therein for a discussion of related results in a more general discrete setting. We will treat the general case here, but use the geometric distribution as a motivating and illustrative example.

In the setting where $X$ is absolutely continuous, the situation is somewhat different. Here we consider instead the number of observations within some threshold $a$ from the maximum. Pakes and Steutel \cite{pakes97} establish a limiting (shifted) mixed Poisson distribution for this quantity as $n\to\infty$, and note that this mixed Poisson simplifies to a geometric distribution in the case where $X$ is in the maximum domain of attraction of a Gumbel distribution. This geometric limit for observations close to the maximum is generalised to a negative binomial limit for observations close to one of the order statistics of our data by Pakes and Li \cite{pakes98}. Our second principal aim in this note is to provide explicit error bounds (again in total variation distance) in the negative binomial approximation of the number of data points close to the order statistics of our sample in this absolutely continuous setting, relevant to applications analogous to those of $K_n$ above. We will formulate this problem more precisely later in this section, after first stating and illustrating our main results in the discrete case. 

All of our approximations and error bounds will be given in the total variation distance, defined for random variables $K$ and $L$ supported on $\mathbb{Z}^+=\{0,1,\ldots\}$ by
\[
d_{\text{TV}}(K,L)=\sup_{E\subseteq\mathbb{Z}^+}|\mathbb{P}(K\in E)-\mathbb{P}(L\in E)|
=\frac{1}{2}\sum_{j\in\mathbb{Z}^+}|\mathbb{P}(K=j)-\mathbb{P}(L=j)|\,.
\]

We use the remainder of this section to present our main results in the discrete setting (Theorems \ref{thm:discrete} and \ref{thm:poisson} below) and absolutely continuous case (in Theorem \ref{thm:continuous}), and to further provide examples to illustrate these results. Proofs of these main results are deferred to Sections \ref{sec:logarithmic}--\ref{sec:poisson}. These proofs make use of Stein's method, adapted here for use with a logarithmic target distribution for the first time and also used to establish new results in the negative binomial approximation of mixed binomial distributions. We will also make use of well-known results on Poisson approximation. Some concluding remarks are given in Section \ref{sec:conc}.   

\subsection{The discrete case}

We begin in the discrete setting, where $X$ takes positive integer values, with mass function $p(j)=\mathbb{P}(X=j)$ and distribution function $F(j)=\mathbb{P}(X\leq j)$ for $j=1,2,\ldots$.  We recall that Lemma 2.1 of Brands \emph{et al$.$} \cite{brands94} gives the probability mass function of the random variable $K_n$ defined in \eqref{eq:kdef}:
\begin{equation}\label{eq:kmass}
\mathbb{P}(K_n=k)=\binom{n}{k}\sum_{j=1}^\infty p(j)^kF(j-1)^{n-k}\,,
\end{equation}
for $k=1,\ldots,n$. 

Although the mass function \eqref{eq:kmass} appears relatively simple, it includes a periodic component that adds complexity to its asymptotics; see Proposition 1 of \cite{kirschenhofer96} for asymptotic results in the case where $X$ has a geometric distribution. This result motivates us to derive approximations in total variation distance for $K_n$, quantifying the error in simple approximations for expectations of bounded functions of $K_n$, not only the mass function and moments as in \cite{kirschenhofer96}.

From \eqref{eq:kmass}, expressions for the factorial moments of $K_n$ follow easily. In particular, letting $(k)_\ell=k(k-1)\cdots(k-\ell+1)$ denote the falling factorial, we have that
\begin{equation}\label{eq:kmean}
\mathbb{E}[(K_n)_\ell]=(n)_\ell\sum_{j=1}^\infty p(j)^\ell F(j)^{n-\ell}\,,
\end{equation}
for $\ell=1,\ldots,n$. 

We will consider the approximation of $K_n$ by both a logarithmic distribution and a Poisson distribution, where we recall that $L\sim\text{L}(\alpha)$ has a logarithmic distribution with parameter $\alpha\in(0,1)$ if
\[
\mathbb{P}(L=k)=\frac{-\alpha^k}{k\log(1-\alpha)}\,,
\]
for $k=1,2,\ldots$, and $Y\sim\text{Pois}(\lambda)$ has a Poisson distribution with parameter $\lambda>0$ if
\[
\mathbb{P}(Y=k)=\frac{e^{-\lambda}\lambda^k}{k!}\,,
\]
for $k=0,1,\ldots$. We begin by considering logarithmic approximation, in which case our main result is the following.
\begin{theorem}\label{thm:discrete}
Let $X_1,\ldots,X_n$ be positive, integer-valued, i$.$i$.$d$.$ random variables with maximum $M_n$. Let $K_n$ be defined as in \eqref{eq:kdef}, with probability mass function as in \eqref{eq:kmass}.
\begin{enumerate}
\item[(a)] Let $L\sim\text{L}(\alpha)$, where $1-\alpha=\frac{\mathbb{P}(K_n=1)}{\mathbb{E}[K_n]}$. Then,
\[
d_\text{TV}(K_n,L)\leq-2n\log(1-\alpha)\sum_{j=1}^\infty p(j)\left(F(j)^{n-1}-\frac{(1-\alpha)(n-1)}{\alpha}p(j)F(j-1)^{n-2}\right)\,.
\]
\item[(b)] Let $L\sim\text{L}(\beta)$, where $1-\beta=\frac{\mathbb{E}[K_n]}{\mathbb{E}[K_n^2]}$. Then,
\[
d_\text{TV}(K_n,L)\leq-2(1+\beta)\log(1-\beta)\mathbb{E}[K_n^2]\left(\beta+(1-\beta)\left[\frac{\mathbb{E}[(K_n)_3]}{\mathbb{E}[(K_n)_2]}-\frac{(n-3)\mathbb{E}[(K_n)_2]}{(n-1)\mathbb{E}[K_n]}\right]\right)\,.
\]
\end{enumerate}
\end{theorem}
The proof of Theorem \ref{thm:discrete}(a) is given in Section \ref{sec:logarithmic} below, while part (b) is proved in Section \ref{sec:negbin}. Note that we typically expect the upper bound of part (a) to be superior to that of part (b); see Remark \ref{rem:thm1} below. We nevertheless include part (b) in the statement of Theorem \ref{thm:discrete}, since there may be examples where it is useful and since its proof uses tools very similar to those needed for Theorems \ref{thm:poisson} and \ref{thm:continuous} below.

Before moving on to state our Poisson approximation results, we illustrate the upper bounds of Theorem \ref{thm:discrete} in our motivating example of geometrically distributed data.
\begin{example}\label{eg:geom1}
We let $X\sim\text{Geom}(p)$ have a geometric distribution with parameter $p\in(0,1)$, that is, $\mathbb{P}(X=j)=p(1-p)^{j-1}$ for $j=1,2,\ldots$. In this example we can check numerically that the upper bound of part (a) of Theorem \ref{thm:discrete} is superior to that of part (b), so we focus on the former result here. In the geometric setting we can easily check that the parameter $\alpha$ of our approximating logarithmic distribution is given by $\alpha=p$: using \eqref{eq:kmass} and \eqref{eq:kmean} we have
\begin{align*}
1-\alpha&=\frac{\sum_{j=1}^\infty p(1-p)^{j-1}[1-(1-p)^{j-1}]^{n-1}}{\sum_{j=1}^\infty p(1-p)^{j-1}[1-(1-p)^j]^{n-1}}\\
&=(1-p)\frac{\sum_{j=2}^\infty p(1-p)^{j-1}[1-(1-p)^{j-1}]^{n-1}}{\sum_{j=1}^\infty p(1-p)^{j}[1-(1-p)^j]^{n-1}}=1-p\,.
\end{align*}
Numerical illustration of the upper bound of Theorem \ref{thm:discrete}(a) in this geometric example is given in Figure \ref{fig:geometric} for the case of $n=20$ and various small-to-moderate values of $p$. We focus on such values of $p$ here since the proof of Theorem \ref{thm:discrete}(a) is geared towards this case: see Remark \ref{rem:small_alpha} below. Other values of $n$ give upper bounds which are very similar to those shown on Figure \ref{fig:geometric}. As expected (see \eqref{eq:rem1} in Remark \ref{rem:small_alpha} below), our upper bound is very close to 
\[
2p\left(\frac{1}{1-p}-(1-p)\right)
\]
for those values of $p$ that we consider here.

We note, however, that even for moderate $p$ there seems to be significant room for improvement in the upper bounds we obtain. For $n=20$ and for several values of $p$ between $0.1$ and $0.2$, we simulated $10^7$ independent copies of $K_n$ for each $p$ and estimated the total variation distance $d_\text{TV}(K_n,L)$, with $L$ as in Theorem \ref{thm:discrete}(a), corresponding to that choice of $p$. In each case we obtained an estimate of order $10^{-5}$, significantly smaller than the upper bounds shown on Figure \ref{fig:geometric} for these values of $p$. 
\begin{figure}
\centering
\includegraphics[scale=0.55]{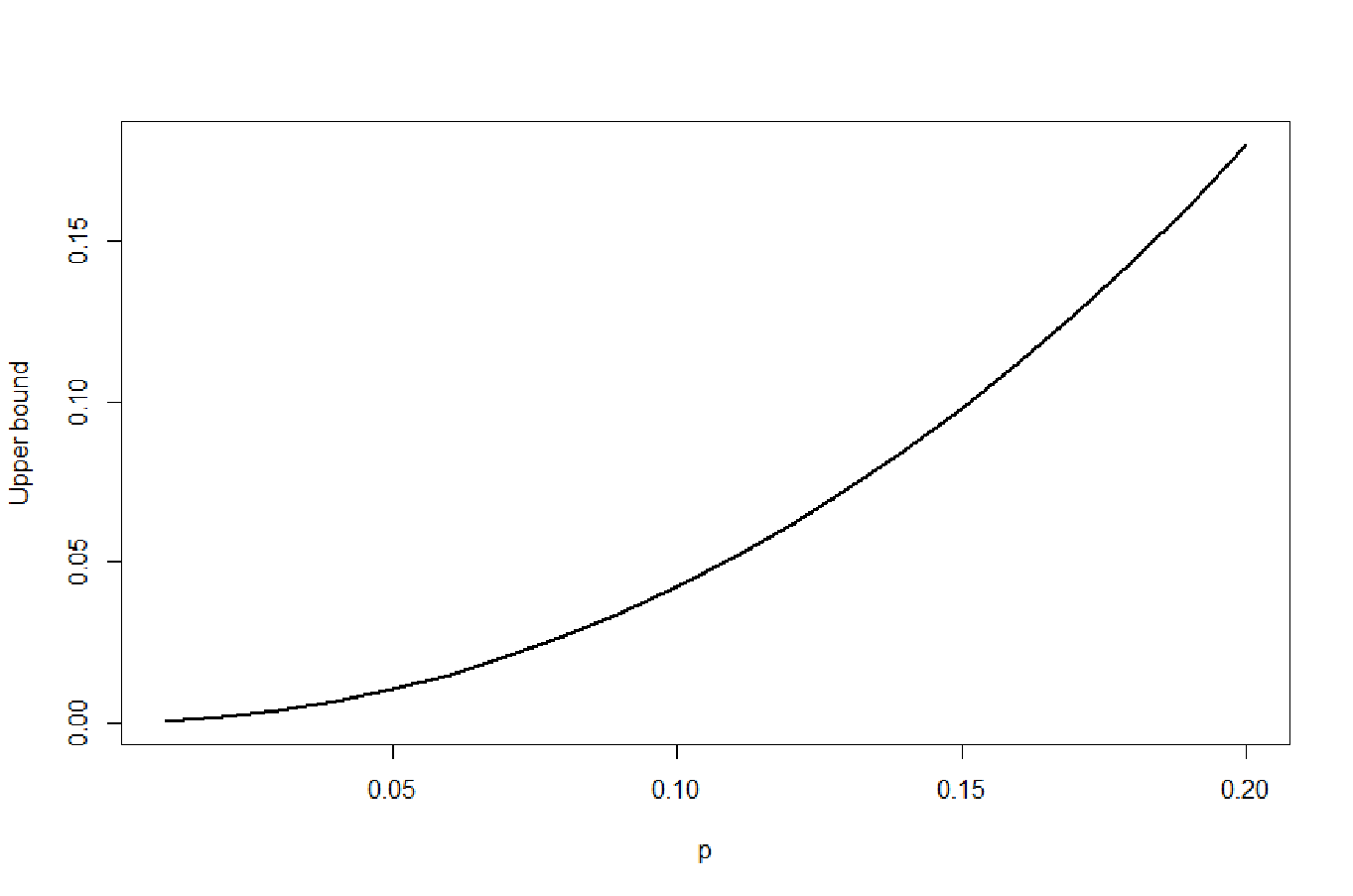}
\caption{The upper bound of Theorem \ref{thm:discrete}(a) in the case where $X$ has a geometric distribution, evaluated for $n=20$ and various values of $p$}
\label{fig:geometric}
\end{figure}
\end{example}

Our main Poisson approximation result in the discrete case is the following.
\begin{theorem}\label{thm:poisson}
Let $X_1,\ldots,X_n$ be positive, integer-valued, i$.$i$.$d$.$ random variables with maximum $M_n$. Let $K_n$ be defined as in \eqref{eq:kdef}, with probability mass function given by \eqref{eq:kmass}. Let $Y\sim\text{Pois}(\lambda)$, where $\lambda=\mathbb{E}[(K_n)_2]/\mathbb{E}[K_n]$. Then,
\begin{multline*}
d_\text{TV}(K_n,Y)\leq\frac{\sqrt{\mathbb{E}[(K_n)_2]-\mathbb{E}[K_n](\mathbb{E}[K_n]-1)}}{2\mathbb{E}[K_n]}+\sqrt{\frac{\mathbb{E}[K_n]}{4\mathbb{E}[(K_n)_2]}}\\
+\frac{(n-1)\mathbb{E}[(K_n)_3]}{(n-2)\mathbb{E}[(K_n)_2]}-\frac{(n-2)\mathbb{E}[(K_n)_2]}{(n-1)\mathbb{E}[K_n]}\,.
\end{multline*}
\end{theorem}

The proof of Theorem \ref{thm:poisson} is given in Section \ref{sec:poisson}. Before moving on to the absolutely continuous case, we again use the example in which $X$ is geometric to illustrate our main result here.
\begin{example}
We again let $X\sim\text{Geom}(p)$ have a geometric distribution, where we take $p=1-\mu/n$ for some suitable $\mu>0$. In this case, Brands \emph{et al$.$} \cite{brands94} (see their Remark 3) show that, as $n\to\infty$, $K_n$ has a defective Poisson limit, with mass $e^{-\mu}$ at infinity. It is therefore reasonable to expect that for (large) fixed $\mu$ and large $n$ we will obtain a good Poisson approximation in this case. We illustrate the upper bound we obtain from Theorem \ref{thm:poisson} in this setting for selected values of $\mu$ and $n$ in Table \ref{tab:geom}.

\begin{table}
\centering
\begin{tabular}{|cc|ccccc|}
\hline
 & & \multicolumn{5}{c|}{$n$} \\
 & & $10^5$ & $10^6$ & $10^7$ & $10^8$ & $10^9$\\
\hline
\multirow{5}{*}{$\mu$} & $100$ & 0.330 & 0.131 & 0.103 & 0.100 & 0.100 \\
& 300 & --- & 0.283 & 0.094 & 0.062 & 0.058\\
& 500 & --- & 0.610 & 0.131 & 0.058 & 0.046\\
& 700 & --- &  ---  & 0.184 & 0.064 & 0.041\\
& 900 & --- &  ---  & 0.251 & 0.073 & 0.039\\
\hline
\end{tabular}
\caption{The upper bound of Theorem \ref{thm:poisson} in the case where $X\sim\text{Geom}(1-\mu/n)$ for selected values of $\mu$ and $n$. The symbol `--' indicates that the upper bound is larger than 1, and therefore uninformative}
\label{tab:geom}
\end{table}
\end{example}

\subsection{The absolutely continuous case}

We turn now to the case where $X$ is a real-valued random variable with probability density function $f(x)$ and distribution function $F(x)=\int_{-\infty}^xf(y)\,\text{d}y$ for $x\in\mathbb{R}$. Following Pakes and Li \cite{pakes98}, we consider the more general setting of negative binomial approximation for the number of observations within a certain distance of one of the order statistics of our sample $X_1,\ldots,X_n$. To that end, we let $X_{1:n}<\cdots<X_{n:n}$ denote the order statistics of our sample, and for $\ell=1,\ldots,n-1$ we define
\begin{equation}\label{eq:kdef2}
K_n(a,\ell)=|\left\{i\in\{1,\ldots,n\}:X_{n-\ell+1:n}-a<X_i\leq X_{n-\ell+1:n}\right\}|\,,
\end{equation}
the number of our $n$ observations that are within a distance $a$ of the $\ell$th order statistic. Pakes and Li (see Theorem 1 of \cite{pakes98}) use a mixed binomial representation of $K_n(a,\ell)-1$ to show that if $X$ is in the maximum domain of attraction of a Gumbel law, then $K_n(a,\ell)-1$ may be well approximated by a negative binomial random variable. We aim to complement this result with explicit error bounds in total variation distance. 

We will write that $Z\sim\text{NB}(\ell,1-\beta)$ has a negative binomial distribution with parameters $\ell>0$ and $\beta\in(0,1)$ if $Z$ satisfies
\[
\mathbb{P}(Z=k)=\frac{\Gamma(\ell+k)}{\Gamma(\ell)k!}(1-\beta)^\ell\beta^k\,,
\]
for $k=0,1,\ldots$, where $\Gamma(\cdot)$ is the gamma function.

Our main result in the absolutely continuous setting is the following error bound in the approximation of $K_n(a,\ell)-1$ by a negative binomial distribution with the same mean.
\begin{theorem}\label{thm:continuous}
Let $X_1,\ldots,X_n$ be real-valued, i$.$i$.$d$.$ random variables as above, and $K_n(a,\ell)$ be as defined in \eqref{eq:kdef2}. Let $Z\sim\text{NB}(\ell,1-\beta)$ with $\beta=\frac{\mathbb{E}[K_n(a,\ell)]-1}{\mathbb{E}[K_n(a,\ell)]+\ell-1}$. Then
\begin{multline*}
d_\text{TV}(K_n(a,\ell)-1,Z)\\
\leq\frac{1-(1-\beta)^\ell}{\beta\ell}\mathbb{E}[K_n(a,\ell)-1]\left(\beta
+(1-\beta)\left[(n-\ell-1)\frac{M_2}{M_1}-(n-\ell-2)M_1\right]\right)\,,
\end{multline*}
where
\[
M_j=n\binom{n-1}{\ell-1}\int_{-\infty}^\infty(1-F(x))^{\ell-1}F(x)^{n-\ell}\left(1-\frac{F(x-a)}{F(x)}\right)^jf(x)\,\text{d}x\,,
\]
for $j=1,2$.
\end{theorem}
Consider the case where $a$ and $\ell$ are fixed (i.e., independent of $n$). The mixed binomial representation of $K_n(a,\ell)-1$ (see Lemma 1 of \cite{pakes98} and Section \ref{sec:pf_cont} below) gives that 
\[
\mathbb{E}[K_n(a,\ell)]-1=(n-\ell)M_1\,,
\]
so that Theorem \ref{thm:continuous} gives an upper bound of order $O(nM_1[1+\frac{nM_2}{M_1}])=O(nM_1+n^2M_2)$. We will see, for example, in Example \ref{eg:gumbel} below that when the $X_i$ have a Gumbel distribution, $M_j=O(n^{-j})$ for $j=1,2$ in this setting, so that our upper bound is of order $O(1)$. We note, however, that in this example and more generally, if $a$ and $\ell$ are allowed to depend on $n$  the order of the upper bound of Theorem \ref{thm:continuous} will vary.

We conclude this section by considering two applications of Theorem \ref{thm:continuous}.

\begin{example}\label{eg:gumbel}
Let $X$ have a Gumbel distribution with density function $f(x)=e^{-x-e^{-x}}$ and distribution function $F(x)=e^{-e^{-x}}$ for $x\in\mathbb{R}$, which has mean equal to $0.57721\ldots$, the Euler--Mascheroni constant, and standard deviation $\pi/\sqrt{6}$. We consider a sample of $n$ i$.$i$.$d$.$ observations from $X$, and establish a geometric approximation bound for the number of these which are within distance $a$ of the sample maximum; that is, we take $\ell=1$ in Theorem \ref{thm:continuous}. Note that the maximum of our $n$ data points has itself a Gumbel distribution, shifted $\log(n)$ units to the right. With this choice of $X$ we have that
\begin{align*}
M_1&=n\int_{-\infty}^\infty(e^{-e^{-x}})^{n-1}\left(1-\frac{e^{-e^{-(x-a)}}}{e^{-e^{-x}}}\right)e^{-x-e^{-x}}\,\text{d}x=\frac{e^a-1}{n+e^a-1}\,,\text{ and}\\
M_2&=n\int_{-\infty}^\infty(e^{-e^{-x}})^{n-1}\left(1-\frac{e^{-e^{-(x-a)}}}{e^{-e^{-x}}}\right)^2e^{-x-e^{-x}}\,\text{d}x=\frac{2(e^a-1)^2}{(n+e^a-1)(n+2e^a-2)}\,.
\end{align*}
Again using the mixed binomial representation of $K_n(a,1)-1$ we have that 
\[
\mathbb{E}[K_n(a,1)]-1=(n-1)M_1=\frac{(n-1)(e^a-1)}{n+e^a-1}\,,
\]
so that $\beta=\frac{(n-1)(e^a-1)}{ne^a}$. With $Z\sim\text{NB}(1,1-\beta)$, Theorem \ref{thm:continuous} then gives
\begin{multline*}
\nonumber d_\text{TV}(K_n(a,1)-1,Z)\leq\frac{(n-1)(e^a-1)}{n+e^a-1}\bigg(\frac{(n-1)(e^a-1)}{ne^a}+\\
\left(1-\frac{(n-1)(e^a-1)}{ne^a}\right)\left[\frac{2(n-2)(e^a-1)}{n+2e^a-2}-\frac{(n-3)(e^a-1)}{n+e^a-1}\right]\bigg)\,,
\end{multline*}
which simplifies to give
\begin{equation}\label{eq:gumbel}
d_\text{TV}(K_n(a,1)-1,Z)\leq\frac{(n-1)(e^a-1)^2}{e^a(n+e^a-1)}\left(1+\frac{n-2}{n+2e^a-2}\right)\,.
\end{equation}
We can see that unfortunately our upper bound is not strong enough to converge to zero as $n\to\infty$ for any fixed $a$, in which setting we know we have a geometric limit thanks to Theorem 1 of \cite{pakes98}. The coupling used in the proof of Theorem \ref{thm:continuous} does not seem to be strong enough to establish this. Nevertheless, our upper bound does tend to zero if $a=a(n)\to0$ as $n\to\infty$, and also gives reasonable numerical bounds on $d_\text{TV}(K_n(a,1)-1,Z)$, even for moderate values of $n$ and $a$. This is illustrated in Figure \ref{fig:gumbel}, which gives the upper bound \eqref{eq:gumbel} in the cases $n=20$ and $n=100$ for various values of $a$. The numerical upper bounds of Figure \ref{fig:gumbel} do, however, leave some room for improvement, overestimating the total variation distance by at least an order of magnitude for the examples illustrated. Using $10^7$ independent simulations of $K_n(a,1)$ for $n=20$ and each $a=0.05,0.1,0.15,0.2$, we estimate $d_\text{TV}(K_n(a,1)-1,Z)$ as in \eqref{eq:gumbel} to be $0.0002,0.0008,0.0014,0.0024$, respectively, for each value of $a$.
\begin{figure}
\centering
\includegraphics[scale=0.55]{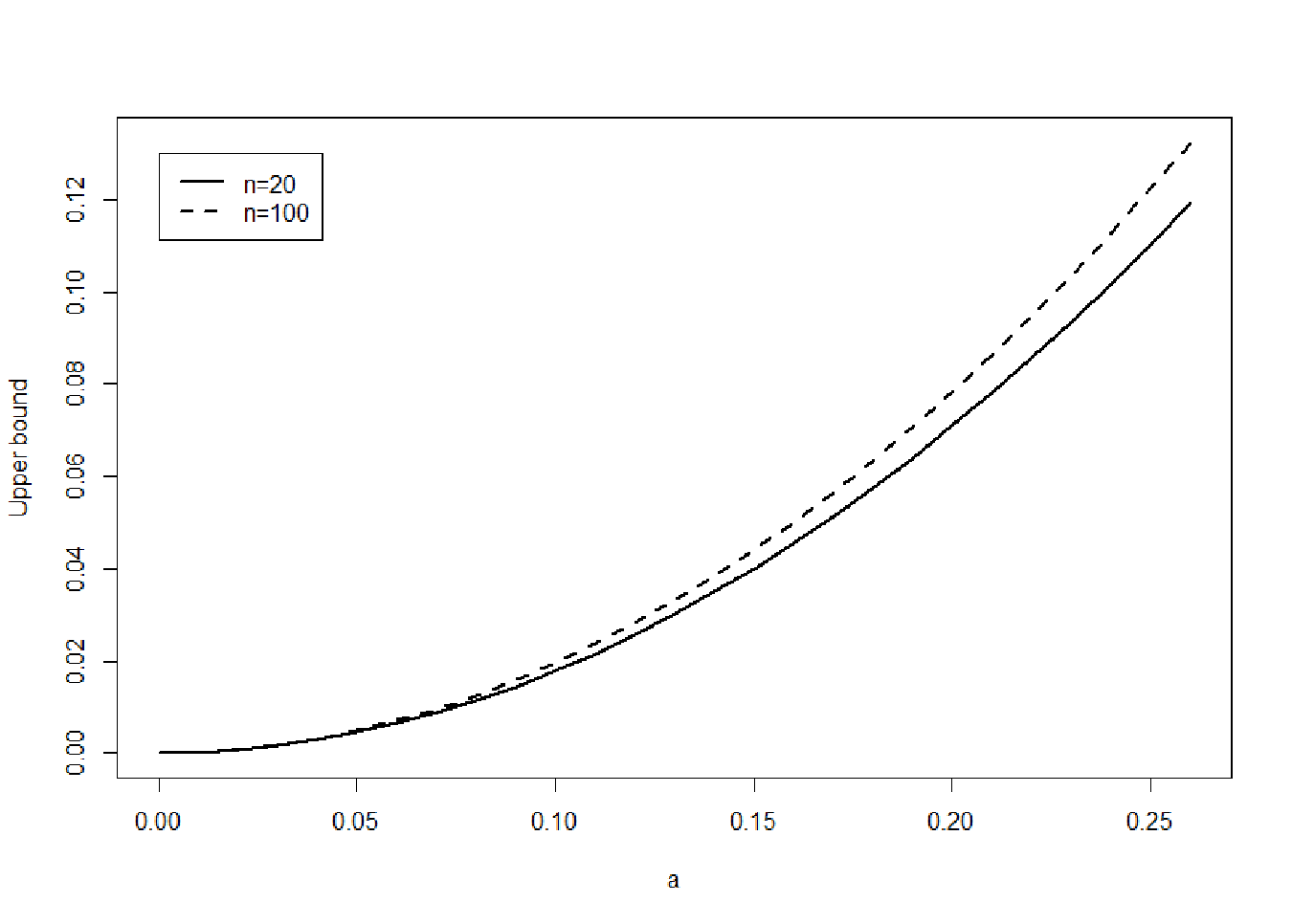}
\caption{The upper bound \eqref{eq:gumbel} for the case where $X$ has a Gumbel distribution, evaluated for various values of $a$ in the cases $n=20$ and $n=100$}
\label{fig:gumbel}
\end{figure}
\end{example}

\begin{example}
Let $X\sim\text{U}(0,1)$ have a uniform distribution on the interval $(0,1)$, so that $f(x)=1$ and $F(x)=x$ for $0\leq x\leq 1$. The uniform distribution is in the maximum domain of attraction of a Weibull (not a Gumbel) distribution, so although we do not expect a negative binomial limit for $K_n(a,\ell)-1$ in the case of fixed $a$ here, we may nevertheless apply our theorem to give explicit upper bounds in the approximation of $K_n(a,\ell)-1$ which may be useful both numerically and in determining conditions on $a=a(n)$ and $\ell=\ell(n)$ under which we may see a distribution which is close to a negative binomial. We assume that $a=a(n)\in(0,1)$ for each $n$, and to keep the exposition concise we will here again assume that $\ell=1$; similar calculations may be carried out for other values of $\ell$. With our choice of $X$ we have that
\[
1-\frac{F(x-a)}{F(x)}=\left\{\begin{array}{ll}
1\,, & x\in(0,a]\,,\\
\frac{a}{x}\,, & x\in(a,1)\,,
\end{array}\right.
\]
so that for $j=1,2$ we have
\[
M_j=n\left(\int_0^ax^{n-1}\,\text{d}x+a^j\int_a^1x^{n-j-1}\,\text{d}x\right)\,.
\]
Hence, $M_1=\frac{an-a^n}{n-1}$ and $M_2=\frac{a^2n-2a^n}{n-2}$. Recall that $\mathbb{E}[K_n(a,1)]-1=(n-1)M_1$. If we choose $a$ such that $an\to0$ as $n\to\infty$, then
\[
\beta=\frac{(n-1)M_1}{(n-1)M_1+1}=\frac{an-a^n}{an-a^n+1}\to0\,,
\]
as $n\to\infty$. From Theorem \ref{thm:continuous} we obtain the upper bound
\begin{align*}
d_\text{TV}(K_n(a,1)-1,Z)
&\leq(n-1)M_1\left(\beta
+(1-\beta)\left[(n-2)\frac{M_2}{M_1}-(n-3)M_1\right]\right)\\
&=\frac{an-a^n}{an-a^n+1}\left(\frac{(n-1)(a^2n-2a^n)}{an-a^n}+\frac{2(an-a^n)}{n-1}\right)\,,
\end{align*}
where $Z\sim\text{NB}(1,1-\beta)$. This upper bound converges to zero as $n\to\infty$ under our assumption that $an\to0$.
\end{example}

In the examples above we are able to explicitly evaluate the integrals $M_j$ needed in the upper bound of Theorem \ref{thm:continuous}. In other examples where analytic expressions for these integrals are not available, they may nevertheless be evaluated numerically.

\section{Stein's method for the logarithmic distribution}\label{sec:logarithmic}

In this section we will prove Theorem \ref{thm:discrete}(a) by developing the tools required to apply Stein's method for probability approximation to a logarithmic target distribution; these tools are also used as part of the proof of Theorem \ref{thm:discrete}(b) in Section \ref{sec:negbin} below, and may be of independent interest. To the best of our knowledge, this is the first use of Stein's technique for approximation by a logarithmic distribution. As a further illustration of our framework, we also give, in Section \ref{sec:LogVsNB}, an error bound in the logarithmic approximation of a negative binomial distribution. We refer the reader to the survey by Ross \cite{ross11} for an introduction to Stein's method, though we will make the exposition here self-contained.

For a non-negative random variable $Y$ with positive mean, we let $Y^\star$ denote the size-biased version of $Y$, defined by
\begin{equation}\label{eq:sb_def}
\mathbb{E}[f(Y^\star)]=\frac{\mathbb{E}[Yf(Y)]}{\mathbb{E}[Y]}
\end{equation}
for all functions $f:\mathbb{R}^+\to\mathbb{R}$ for which the expectation on the right-hand side exists. In the special case where $Y$ takes non-negative integer values, $Y^\star$ may equivalently be defined by writing
\[
\mathbb{P}(Y^\star=k)=\frac{k\mathbb{P}(Y=k)}{\mathbb{E}[Y]}
\]
for $k=0,1,\ldots$. For a logarithmic random variable $L\sim\text{L}(\alpha)$ with mean $\mathbb{E}[L]=\frac{-\alpha}{(1-\alpha)\log(1-\alpha)}$ and a bounded function $f:\mathbb{Z}^+\to\mathbb{R}$ we may write
\begin{align*}
\mathbb{E}f(L^\star-1)&=(1-\alpha)\sum_{k=0}^\infty\alpha^kf(k)\\
&=(1-\alpha)f(0)-(1-\alpha)\log(1-\alpha)\sum_{k=1}^\infty kf(k)\mathbb{P}(L=k)\\
&=(1-\alpha)f(0)+\frac{\alpha}{\mathbb{E}[L]}\sum_{k=1}^\infty kf(k)\mathbb{P}(L=k)
=(1-\alpha)f(0)+\alpha\mathbb{E}[f(L^\star)]\,.
\end{align*}
That is, $L^\star-1$ is equal in distribution to $I_\alpha L^\star$, where $I_\alpha$ has a Bernoulli distribution with mean $\alpha$ and is independent of $L^\star$.

This leads us to define the following Stein equation for a logarithmic target distribution: for a given function $h:\mathbb{Z}^+\to\mathbb{R}$ with $\mathbb{E}[h(L)]=0$ we let $f_h:\mathbb{Z}^+\to\mathbb{R}$ be such that $f_h(0)=0$ and 
\begin{equation}\label{eq:LogSteinEq}
h(k)=kf_h(k-1)-\alpha kf_h(k)
\end{equation}
for $k=1,2,\ldots$. This is motivated by noting that replacing $k$ by our logarithmic random variable $L$ and taking expectations on the right-hand side gives $\mathbb{E}f_h(L^\star-1)-\alpha\mathbb{E}[f_h(L^\star)]$. 

By replacing $k$ by a random variable $K$ and taking expectations in our Stein equation \eqref{eq:LogSteinEq}, we may then write
\begin{equation}\label{eq:LogTV}
d_\text{TV}(K,L)=\sup_{h\in\mathcal{H}}|\mathbb{E}[h(K)]|
=\sup_{h\in\mathcal{H}}|\mathbb{E}[Kf_h(K-1)-\alpha Kf_h(K)]|\,,
\end{equation}
where 
\begin{equation}\label{eq:TVset}
\mathcal{H}=\{I(\cdot\in E)-\mathbb{P}(L\in E):E\subseteq\mathbb{Z}^+\}
\end{equation} 
and $I$ denotes an indicator function. In order to bound this final expression, we will need to control the behaviour of $f_h$ for $h\in\mathcal{H}$. It is straightforward to check that the solution to \eqref{eq:LogSteinEq} is given by
\begin{align*}
f_h(k)&=-\frac{1}{\alpha^{k+1}}\sum_{j=1}^kh(j)\frac{\alpha^{j}}{j}
=\frac{\log(1-\alpha)}{\alpha^{k+1}}\sum_{j=1}^kh(j)\mathbb{P}(L=j)\\
&=-\frac{\log(1-\alpha)}{\alpha^{k+1}}\sum_{j=k+1}^\infty h(j)\mathbb{P}(L=j)
=\frac{1}{\alpha}\sum_{j=k+1}^\infty h(j)\frac{\alpha^{j-k}}{j}\\
&=\frac{1}{\alpha}\sum_{j=1}^\infty h(j+k)\frac{\alpha^j}{j+k}\,,
\end{align*}
where we use the fact that $\mathbb{E}[h(L)]=0$ for $h\in\mathcal{H}$. Hence, since $|h(x)|\leq1$ for $h\in\mathcal{H}$,
\begin{equation}\label{eq:LogSteinFactor}
|f_h(k)|\leq\frac{1}{\alpha}\sum_{j=1}^\infty\frac{\alpha^j}{j+k}\leq\frac{1}{\alpha}\sum_{j=1}^\infty\frac{\alpha^j}{j}=-\frac{\log(1-\alpha)}{\alpha}\,.
\end{equation}

\subsection{Logarithmic approximation for the negative binomial distribution}\label{sec:LogVsNB}

Before proceeding to the proof of Theorem \ref{thm:discrete}(a), we illustrate the framework we have set up with a simple example: a bound on the total variation distance between a negative binomial distribution and a logarithmic distribution.
\begin{example}
Let $L\sim\text{L}(\alpha)$ for some $\alpha\in(0,1)$, and $Z\sim\text{NB}(\ell,1-\beta)$ for some $\ell>0$ and $\beta\in(0,1)$. Following \eqref{eq:LogTV}, in order to bound the total variation distance between $Z$ and $L$ we bound $\mathbb{E}[Zf_h(Z-1)-\alpha Zf_h(Z)]$ for $h\in\mathcal{H}$. We begin by noting that for all bounded functions $g:\mathbb{Z}^+\to\mathbb{R}$, 
\[
\beta\mathbb{E}[(\ell+Z)g(Z+1)]=\mathbb{E}[Zg(Z)]\,;
\]
see Lemma 1 of \cite{brown99}. Hence,
\begin{align*}
\mathbb{E}[Zf_h(Z-1)-\alpha Zf_h(Z)]&=\mathbb{E}[(\beta\ell-(\alpha-\beta)Z)f_h(Z)]\\
&=\frac{(1-\alpha)\beta\ell}{1-\beta}\mathbb{E}[f_h(Z)]+(\alpha-\beta)\mathbb{E}[(\mathbb{E}[Z]-Z)f_h(Z)]\,,
\end{align*}
since $\mathbb{E}[Z]=\frac{\beta\ell}{1-\beta}$. Hence, using the bound \eqref{eq:LogSteinFactor} we have that, for $h\in\mathcal{H}$,
\begin{align*}
|\mathbb{E}[Zf_h(Z-1)-\alpha Zf_h(Z)]|&\leq-\frac{(1-\alpha)\beta\ell}{\alpha(1-\beta)}\log(1-\alpha)-\frac{(\alpha-\beta)}{\alpha}\log(1-\alpha)\sqrt{\text{Var}(Z)}\\
&=-\frac{\log(1-\alpha)\sqrt{\beta\ell}}{\alpha(1-\beta)}\left((1-\alpha)\sqrt{\beta\ell}+\alpha-\beta\right)\,,
\end{align*}
since $\text{Var}(Z)=\frac{\beta\ell}{(1-\beta)^2}$. Hence,
\[
d_\text{TV}(Z,L)\leq-\frac{\log(1-\alpha)\sqrt{\beta\ell}}{\alpha(1-\beta)}\left((1-\alpha)\sqrt{\beta\ell}+\alpha-\beta\right)\,.
\]
In particular, with the choice $\alpha=\beta$ we obtain $d_\text{TV}(Z,L)\leq-\log(1-\alpha)\ell\to0$ as $\ell\to0$, as expected.
\end{example}

\subsection{Proof of Theorem \ref{thm:discrete}(a)}

To establish Theorem \ref{thm:discrete}(a), we combine the mass function and moments given in \eqref{eq:kmass} and \eqref{eq:kmean}, respectively, with the following result giving an upper bound on the total variation distance between a positive, integer-valued random variable $K$ and $L\sim\text{L}(\alpha)$, where we choose $\alpha=\mathbb{P}(K^\star>1)$. This choice is convenient in the proof (as we will see following \eqref{eq:LogPf1} below), and also gives
\[
1-\alpha=\frac{\mathbb{P}(K=1)}{\mathbb{E}[K]}=\frac{\mathbb{P}(L=1)}{\mathbb{E}[L]}\,.
\]
\begin{theorem}\label{thm:log}
Let $K$ be a positive, integer-valued random variable with finite mean and $L\sim\text{L}(\alpha)$, where $\alpha=\mathbb{P}(K^\star>1)$. Then,
\[
d_\text{TV}(K,L)\leq-2\log(1-\alpha)\left(\mathbb{E}[K]-\frac{2(1-\alpha)}{\alpha}\mathbb{P}(K=2)\right)\,.
\]
\end{theorem}
\begin{proof}
We again use \eqref{eq:LogTV} and thus need to bound 
\[
\mathbb{E}[Kf_h(K-1)-\alpha Kf_h(K)]=\mathbb{E}[K]\mathbb{E}[f_h(\widetilde{K}^\star-1)-f_h(I_\alpha K^\star)]\,,
\]
where, as above, $I_\alpha$ is a Bernoulli random variable with mean $\alpha$ independent of $K^\star$, and $\widetilde{K}^\star$ is an independent copy of $K^\star$. Using the bound \eqref{eq:LogSteinFactor}, for $h\in\mathcal{H}$ we may then write
\[
|\mathbb{E}[f_h(\widetilde{K}^\star-1)-f_h(I_\alpha K^\star)]|\leq-\frac{2\log(1-\alpha)}{\alpha}\mathbb{P}(\widetilde{K}^\star-1\not=I_\alpha K^\star)
\]
for any coupling of $\widetilde{K}^\star$ and $I_\alpha K^\star$, so that
\begin{equation}\label{eq:LogPf1}
d_\text{TV}(K,L)\leq-\frac{2\log(1-\alpha)}{\alpha}\mathbb{E}[K]\mathbb{P}(\widetilde{K}^\star-1\not=I_\alpha K^\star)\,.
\end{equation}
Recalling that $\alpha=\mathbb{P}(\widetilde{K}^\star>1)$, we may construct $I_\alpha$ as the indicator of the event that $\widetilde{K}^\star>1$. With this choice,
\[
\mathbb{P}(\widetilde{K}^\star-1\not=I_\alpha K^\star)=\alpha\mathbb{P}(\widetilde{K}^\star-1\not=K^\star|\widetilde{K}^\star>1)
\]
and 
\begin{align}
\nonumber\mathbb{P}(\widetilde{K}^\star-1=K^\star|\widetilde{K}^\star>1)&\geq\mathbb{P}(\widetilde{K}^\star=2,K^\star=1|\widetilde{K}^\star>1)\\
\nonumber&=\mathbb{P}(\widetilde{K}^\star=2|\widetilde{K}^\star>1)\mathbb{P}(K^\star=1)\\
\label{eq:kproof}&=\frac{1-\alpha}{\alpha}\mathbb{P}(K^\star=2)
=\frac{2(1-\alpha)}{\alpha\mathbb{E}[K]}\mathbb{P}(K=2)\,,
\end{align}
so that 
\begin{equation}\label{eq:LogPf2}
\mathbb{P}(\widetilde{K}^\star-1\not=I_\alpha K^\star)\leq\alpha\left(1-\frac{2(1-\alpha)}{\alpha\mathbb{E}[K]}\mathbb{P}(K=2)\right)\,.
\end{equation}
Our result then follows on combining \eqref{eq:LogPf1} and \eqref{eq:LogPf2}.
\end{proof}
\begin{remark}\label{rem:small_alpha}
When bounding the probability $\mathbb{P}(\widetilde{K}^\star-1=K^\star|\widetilde{K}^\star>1)$ in \eqref{eq:kproof} we could, of course, include further terms of the form $\mathbb{P}(\widetilde{K}^\star=j+1,K^\star=j|\widetilde{K}^\star>1)$ for $j>1$ to get a more precise lower bound. However, in our coupling construction we have in mind the setting in which $\alpha$ is small, so that both $\widetilde{K}^\star-1$ and $I_\alpha K^\star$ are equal to zero with high probability. In this setting, further terms of this sort for $j>1$ are significantly smaller than the $j=1$ term we have included in the bound \eqref{eq:kproof}, so add complexity to the resulting bound on total variation distance for no significant gain in accuracy. Related to this, we note that that even if $K$ is logarithmically distributed, the upper bound of Theorem \ref{thm:log} is not zero. Using $\mathbb{E}[L]=\frac{-\alpha}{(1-\alpha)\log(1-\alpha)}$ and $\mathbb{P}(L=2)=\frac{-\alpha^2}{2\log(1-\alpha)}$, the upper bound of Theorem \ref{thm:log} with $K=L$ is equal to 
\begin{equation}\label{eq:rem1}
2\alpha\left(\frac{1}{1-\alpha}-(1-\alpha)\right)\,,
\end{equation}
which is increasing in $\alpha$.
\end{remark}

\section{Negative binomial approximation for mixed binomials}\label{sec:negbin}

Our main aim in this section is to prove Theorems \ref{thm:discrete}(b) and \ref{thm:continuous}. To do this, we begin by considering the more general setting of negative binomial approximation for mixed binomial random variables. We let $W$ have a mixed binomial distribution. In particular, for some $n\in\{1,2,\ldots\}$, $\ell\in\{1,\ldots,n-1\}$ and for some random variable $Q$ supported on $[0,1]$, we let the conditional distribution $W|Q\sim\text{Bin}(n-\ell,Q)$ have a binomial distribution with parameters $n-\ell$ and $Q$. We then write that $W\sim\text{MBin}(n-\ell,Q)$. We consider the approximation of $W$ by the negative binomial random variable $Z\sim\text{NB}(\ell,1-\beta)$, where we choose $\beta=\frac{\mathbb{E}[W]}{\mathbb{E}[W]+\ell}$ so that $\mathbb{E}[Z]=\mathbb{E}[W]$. Our main result here is the following.
\begin{theorem}\label{thm:negbin}
Let $W|Q\sim\text{Bin}(n-\ell,Q)$ and $Z\sim\text{NB}(\ell,1-\beta)$ be as above, with $\beta=\frac{\mathbb{E}[W]}{\mathbb{E}[W]+\ell}$. Then
\[
d_\text{TV}(W,Z)
\leq\frac{1-(1-\beta)^\ell}{\beta\ell}\mathbb{E}[W]\left(\beta+(1-\beta)\left[(n-\ell-1)\frac{\mathbb{E}[Q^2]}{\mathbb{E}[Q]}-(n-\ell-2)\mathbb{E}[Q]\right]\right)\,.
\]
\end{theorem}
\begin{proof}
We use the framework of Stein's method for negative binomial approximation established by Brown and Phillips \cite{brown99}; see also \cite{ross13} for more recent developments. To that end, with $h\in\mathcal{H}$ as defined in \eqref{eq:TVset}, we let $g_h:\mathbb{Z}^+\to\mathbb{R}$ satisfy $g_h(0)=0$ and
\[
h(k)=\beta(\ell+k)g_h(k+1)-kg_h(k)\,,
\]
for $k\in\mathbb{Z}^+$, so that we may write
\begin{equation}\label{eq:NegBinPf1}
d_\text{TV}(W,Z)=\sup_{h\in\mathcal{H}}|\mathbb{E}[\beta(\ell+W)g_h(W+1)-Wg_h(W)]|\,.
\end{equation}
From Lemma 5 of \cite{brown99} we have that
\begin{equation}\label{eq:NegBinStinFac}
\sup_{h\in\mathcal{H}}\sup_{k\in\mathbb{Z}^+}|\Delta g_h(k)|\leq\frac{1-(1-\beta)^\ell}{\beta\ell}\,,
\end{equation}
where $\Delta g_h(k)=g_h(k+1)-g_h(k)$.

Now, with our choice of $\beta$ we have $\mathbb{E}[W]=\frac{\beta\ell}{1-\beta}$ and
\begin{align*}
\frac{\beta\mathbb{E}[(\ell+W)g_h(W+1)]}{\mathbb{E}[W]}&=\frac{(1-\beta)(\ell-1)}{\ell}\mathbb{E}[g_h(W+1)]\\
&\qquad\qquad\qquad+\left(1-\frac{(1-\beta)(\ell-1)}{\ell}\right)\frac{\mathbb{E}[(W+1)g_h(W+1)]}{\mathbb{E}[W]+1}\\
&=\frac{(1-\beta)(\ell-1)}{\ell}\mathbb{E}[g_h(W+1)]\\
&\qquad\qquad\qquad+\left(1-\frac{(1-\beta)(\ell-1)}{\ell}\right)\mathbb{E}[g_h((W+1)^\star)]\,,
\end{align*}
where the final equality uses the definition \eqref{eq:sb_def} of size-biasing. Hence,
\begin{align*}
\mathbb{E}[\beta(\ell+W)g_h(W+1)]&=\mathbb{E}[W]\left\{(1-\gamma)\mathbb{E}[g_h(W+1)]+\gamma\mathbb{E}[g_h((W+1)^\star)]\right\}\\
&=\mathbb{E}[W]\mathbb{E}\left[g_h((1-I_\gamma)(W+1)+I_\gamma (W+1)^\star)\right]\,,
\end{align*}
where $\gamma=1-\frac{(1-\beta)(\ell-1)}{\ell}$ and $I_\gamma$ is a Bernoulli random variable with mean $\gamma$ independent of all other random variables with which it appears. Again using the definition of size-biasing, we may therefore write
\begin{equation}\label{eq:NegBinPf2}
\mathbb{E}[\beta(\ell+W)g_h(W+1)-Wg_h(W)]=\mathbb{E}[W]\mathbb{E}[g_h(W^\prime)-g_h(W^\star)]\,,
\end{equation}
where $W^\prime$ is equal in distribution to $(1-I_\gamma)(W+1)+I_\gamma(W+1)^\star$.

It is well-known that we may size-bias a sum of independent random variables by selecting one of these random variables with probabilities proportional to their means, and replacing the chosen random variable with its size-biased version; see Section 2.4 of \cite{arratia19}. Combining this with Lemma 2.4 of \cite{arratia19} on size-biasing a mixture, we have that $W^\star=1+A$, where $A\sim\text{MBin}(n-\ell-1,Q^\star)$ and a mixed binomial distribution with first parameter equal to zero should be interpreted as a point mass at zero. Similarly, and noting that $\frac{\mathbb{E}[W]}{\mathbb{E}[W]+1}=\frac{\beta}{\gamma}$, we may write
\[
(W+1)^\star=I_{\beta/\gamma}(W^\star+1)+(1-I_{\beta/\gamma})(W+1)=1+I_{\beta/\gamma}W^\star+(1-I_{\beta/\gamma})W\,,
\]
where $I_{\beta/\gamma}$ is a Bernoulli random variable with mean $\beta/\gamma$ which is independent of all else. It then follows that we may write
\[
W^\prime=1+\left[1-I_\gamma I_{\beta/\gamma}\right]W+I_\gamma I_{\beta/\gamma}W^\star=1+I_\gamma I_{\beta/\gamma}+B+C\,,\
\]
where $B\sim\text{MBin}(n-\ell-1,Q^+)$,
\[
Q^+=\left(1-I_\gamma I_{\beta/\gamma}\right)Q+I_\gamma I_{\beta/\gamma}Q^\star\,,
\]
and $C\sim\text{MBin}(1,(1-I_\gamma I_{\beta/\gamma})Q)$, and where, conditional on $Q$, $Q^\star$ and the Bernoulli random variables $I_\gamma$ and $I_{\beta/\gamma}$, $B$ and $C$ are independent. Since $Q^\star$ is stochastically larger than $Q$, the closure of the usual stochastic order under mixtures (see Theorem 1.A.3(d) of \cite{shaked07}) implies that $Q^\star$ is stochastically larger than $Q^+$. Since the binomial distribution $\text{Bin}(n,q)$ is stochastically increasing in $q$ (see Example 1.A.25 of \cite{shaked07}), and again using closure of stochastic ordering under mixtures, we then have that $A$ is stochastically larger than $B$. Hence, for $h\in\mathcal{H}$, the bound \eqref{eq:NegBinStinFac} gives
\begin{align}
\nonumber|\mathbb{E}[g_h(W^\prime)-g_h(W^\star)]|&\leq\frac{1-(1-\beta)^\ell}{\beta\ell}\mathbb{E}|W^\prime-W^\star|\leq\frac{1-(1-\beta)^\ell}{\beta\ell}\left(\beta+\mathbb{E}[A-B+C]\right)\\
\nonumber&=\frac{1-(1-\beta)^\ell}{\beta\ell}\left(\beta+(n-\ell-1)\mathbb{E}[Q^\star-Q^+]+(1-\beta)\mathbb{E}[Q]\right)\\
\nonumber&=\frac{1-(1-\beta)^\ell}{\beta\ell}\left(\beta+(1-\beta)\left\{(n-\ell-1)\mathbb{E}[Q^\star-Q]+\mathbb{E}[Q]\right\}\right)\\
\label{eq:NegBinPf3}&=\frac{1-(1-\beta)^\ell}{\beta\ell}(\beta+(1-\beta)\{(n-\ell-1)\mathbb{E}[Q^\star]-(n-\ell-2)\mathbb{E}[Q]\})\,.
\end{align}
From the definition \eqref{eq:sb_def} we have that $\mathbb{E}[Q^\star]=\frac{\mathbb{E}[Q^2]}{\mathbb{E}[Q]}$. The desired result then follows by combining this with the representation \eqref{eq:NegBinPf1} and \eqref{eq:NegBinPf2}--\eqref{eq:NegBinPf3}.
\end{proof}

\subsection{Proof of Theorem \ref{thm:discrete}(b)}

We now apply Theorem \ref{thm:negbin} to establish the upper bound in Theorem \ref{thm:discrete}(b). To that end, we will first transform the problem of logarithmic approximation for a positive, integer-valued random variable $K$ to the problem of geometric approximation for the size-biased version $K^\star$, as defined in \eqref{eq:sb_def}.
\begin{lemma}\label{lem:geom}
Let $K$ be a positive, integer-valued random variable with finite mean, $L\sim\text{L}(\beta)$ for some $\beta\in(0,1)$, and $G\sim\text{Geom}(1-\beta)$ have a geometric distribution with parameter $1-\beta$. Then
\[
d_\text{TV}(K,L)\leq-\frac{2(1+\beta)\log(1-\beta)}{\beta}\mathbb{E}[K]d_\text{TV}(K^\star,G)\,.
\]
\end{lemma}
\begin{proof}
We use the framework for Stein's method for logarithmic approximation developed in Section \ref{sec:logarithmic}. Arguing as we did for \eqref{eq:LogPf1}, we have that for any coupling of $K^\star-1$ and $I_\beta K^\star$, where $I_\beta$ is a Bernoulli random variable with mean $\beta$ independent of the $K^\star$ it multiplies here,
\[
d_\text{TV}(K,L)\leq-\frac{2\log(1-\beta)}{\beta}\mathbb{E}[K]\mathbb{P}(K^\star-1\not=I_\beta K^\star)\,.
\]
Recalling that we may equivalently define the total variation distance as $d_\text{TV}(K^\star-1,I_\beta K^\star)=\inf\mathbb{P}(K^\star-1\not=I_\beta K^\star)$, where the infimum is taken over all couplings of $K^\star-1$ and $I_\beta K^\star$, we may thus write
\begin{equation}\label{eq:LemPf}
d_\text{TV}(K,L)\leq-\frac{2\log(1-\beta)}{\beta}\mathbb{E}[K]d_\text{TV}(K^\star-1,I_\beta K^\star)\,.
\end{equation}
From the definition of the geometric distribution we can easily check that $G-1$ is equal in distribution to $I_\beta G$. We may thus use the triangle inequality for total variation distance to write
\begin{align*}
d_\text{TV}(K^\star-1,I_\beta K^\star)&\leq d_\text{TV}(K^\star-1,G-1)+d_\text{TV}(I_\beta K^\star,I_\beta G)\\
&=d_\text{TV}(K^\star-1,G-1)+\beta d_\text{TV}(K^\star,G)\\
&=(1+\beta)d_\text{TV}(K^\star,G)\,.
\end{align*}
Combining this with \eqref{eq:LemPf} yields the desired result.
\end{proof}
To establish Theorem \ref{thm:discrete}(b) we now take $K=K_n$ in the setting of Lemma \ref{lem:geom}. With notation as in that latter result, we note that $G-1\sim\text{NB}(1,1-\beta)$ and so we may apply Theorem \ref{thm:negbin} with $\ell=1$ to bound $d_\text{TV}(K_n^\star,G)$ once we have shown that $K_n^\star-1$ is a mixed binomial random variable. To that end, we give the following explicit construction of $K_n^\star-1$: 
\begin{itemize}
\item Sample the random variable $M$ according to the distribution 
\[
\mathbb{P}(M=m)=\frac{\mathbb{P}(X=m)\mathbb{P}(X\leq m)^{n-1}}{\sum_{j=1}^\infty\mathbb{P}(X=j)\mathbb{P}(X\leq j)^{n-1}}\,,
\]
for $m=1,2,\ldots$.
\item Set $X_n=M$, and sample $X_1^\prime,\ldots,X_{n-1}^\prime$ independently, each according to the distribution
\[
\mathbb{P}(X_1^\prime=k|M=m)=\frac{\mathbb{P}(X=k)}{\mathbb{P}(X\leq m)}\,,
\]
for $k=1,\ldots,m$.
\item Set $K_n^\star-1=|\left\{i\in\{1,\ldots,n-1\}:X_i^\prime=M\right\}|$\,.
\end{itemize}
Using an argument analogous to that of Lemma 2.1 of Brands \emph{et al$.$} \cite{brands94}, we then have that
\[
\mathbb{P}(K_n^\star=k|M=m)=\binom{n-1}{k-1}\frac{\mathbb{P}(X=m)^{k-1}\mathbb{P}(X\leq m-1)^{n-k}}{\mathbb{P}(X\leq m)^{n-1}}\,,
\]
so that 
\begin{equation}\label{eq:MixBinSB}
K_n^\star-1|M\sim\text{Bin}(n-1,q(M))\,,
\end{equation}
where $q(m)=\frac{\mathbb{P}(X=m)}{\mathbb{P}(X\leq m)}$. We will return to this mixed binomial representation of $K_n^\star-1$ in the proof of Theorem \ref{thm:poisson} in Section \ref{sec:poisson} below. Removing the conditioning we further have that
\begin{align*}
\mathbb{P}(K_n^\star=k)&=\sum_{m=1}^\infty\mathbb{P}(M=m)\mathbb{P}(K_n^\star=k|M=m)\\
&=\binom{n-1}{k-1}\frac{\sum_{m=1}^\infty\mathbb{P}(X=m)^k\mathbb{P}(X\leq m-1)^{n-k}}{\sum_{j=1}^\infty\mathbb{P}(X=j)\mathbb{P}(X\leq j)^{n-1}}
=\frac{k\mathbb{P}(K_n=k)}{\mathbb{E}[K_n]}\,,
\end{align*}
using the mass function and first moment given in \eqref{eq:kmass} and \eqref{eq:kmean}, respectively, for the final inequality, confirming that $K_n^\star$ is indeed a size-biased version of $K_n$. 

With $K_n^\star-1|M\sim\text{Bin}(n-1,q(M))$, and noting that $d_\text{TV}(K_n^\star,G)=d_\text{TV}(K_n^\star-1,G-1)$, Theorem \ref{thm:negbin} then gives us that 
\[
d_\text{TV}(K_n^\star,G)
\leq\mathbb{E}[K_n^\star-1]\left(\beta+(1-\beta)\left[(n-2)\frac{\mathbb{E}[q(M)^2]}{\mathbb{E}[q(M)]}-(n-3)\mathbb{E}[q(M)]\right]\right)\,,
\]
where we note that the choice $1-\beta=\frac{1}{\mathbb{E}[K_n^\star]}=\frac{\mathbb{E}[K_n]}{\mathbb{E}[K_n^2]}$ in Theorem \ref{thm:negbin} matches that in Theorem \ref{thm:discrete}(b). Combining this with Lemma \ref{lem:geom} we then have that
\begin{multline*}
d_\text{TV}(K,L)
\leq-\frac{2(1+\beta)\log(1-\beta)}{\beta}\mathbb{E}[K_n]\mathbb{E}[K_n^\star-1]\\
\times\left(\beta+(1-\beta)\left[(n-2)\frac{\mathbb{E}[q(M)^2]}{\mathbb{E}[q(M)]}-(n-3)\mathbb{E}[q(M)]\right]\right)\,.
\end{multline*}
The proof of Theorem \ref{thm:discrete}(b) is completed upon using the choice of $\beta$ to note that 
\[
\frac{1}{\beta}\mathbb{E}[K_n]\mathbb{E}[K_n^*-1]=\mathbb{E}[K_n]\mathbb{E}[K_n^\star]=\mathbb{E}[K_n^2]\,,
\]
and using \eqref{eq:kmean} to note that
\begin{align}
\nonumber\mathbb{E}[q(M)]&=\frac{\sum_{m=1}^\infty\mathbb{P}(X=m)^2\mathbb{P}(X\leq m)^{n-2}}{\sum_{j=1}^\infty\mathbb{P}(X=j)\mathbb{P}(X\leq j)^{n-1}}=\frac{\mathbb{E}[(K_n)_2]}{(n-1)\mathbb{E}[K_n]}\,,\text{ and }\\
\label{eq:QMoments}\mathbb{E}[q(M)^2]&=\frac{\sum_{m=1}^\infty\mathbb{P}(X=m)^3\mathbb{P}(X\leq m)^{n-3}}{\sum_{j=1}^\infty\mathbb{P}(X=j)\mathbb{P}(X\leq j)^{n-1}}=\frac{\mathbb{E}[(K_n)_3]}{(n-1)(n-2)\mathbb{E}[K_n]}\,.
\end{align}
\begin{remark}\label{rem:thm1}
By considering how the upper bound of Theorem \ref{thm:discrete}(b) behaves should the first three moments of $K_n$ match those of a logarithmic distribution, we would expect this to be generally inferior to Theorem \ref{thm:discrete}(a). Using the fact that
\[
\mathbb{E}[(L)_\ell]=\frac{-(\ell-1)!}{\log(1-\alpha)}\left(\frac{\alpha}{1-\alpha}\right)^\ell
\]
for $L\sim\text{L}(\alpha)$ and $\ell=1,2,\ldots$, with $\beta=\alpha$ and the factorial moments of $K_n$ replaced by those of $L$, the upper bound of Theorem \ref{thm:discrete}(b) becomes
\[
\frac{2\alpha(1+\alpha)}{(1-\alpha)^2}\left(\alpha+(1-\alpha)\left[\frac{2\alpha}{1-\alpha}-\frac{(n-3)\alpha}{(n-1)(1-\alpha)}\right]\right)\,,
\]
which for large $n$ is approximately equal to $\frac{4\alpha^2(1+\alpha)}{(1-\alpha)^2}$. For all $\alpha\in(0,1)$ this is greater than \eqref{eq:rem1}, which gives a comparable expression for the bound of Theorem \ref{thm:discrete}(a). Although we typically expect that part (a) of Theorem \ref{thm:discrete} is superior to part (b), we nevertheless include both since we have been unable to show that this is the case for all possible choices of the underlying distribution of the $X_i$. 
\end{remark}

\subsection{Proof of Theorem \ref{thm:continuous}}\label{sec:pf_cont}

Finally in this section, we use Theorem \ref{thm:negbin} to establish Theorem \ref{thm:continuous}. By Lemma 1 of Pakes and Li \cite{pakes98}, we have that $K_n(a,\ell)-1\sim\text{MBin}(n-\ell,r_a(X_{n-\ell+1:n}))$, where $r_a(x)=1-\frac{F(x-a)}{F(x)}$ for $x\in\mathbb{R}$. We note that $X_{n-\ell+1:n}$ has distribution function
\[
\mathbb{P}(X_{n-\ell+1:n}\leq x)=n\binom{n-1}{\ell-1}\int_{-\infty}^x(1-F(y))^{\ell-1}F(y)^{n-\ell}f(y)\,\text{d}y
\]
and density function
\[
f_\ell(x)=n\binom{n-1}{\ell-1}(1-F(x))^{\ell-1}F(x)^{n-\ell}f(x)\,,
\]
for $x\in\mathbb{R}$, from which it follows that 
\[
\mathbb{E}[r_a(X_{n-\ell+1:n})^j]
=n\binom{n-1}{\ell-1}\int_{-\infty}^\infty(1-F(x))^{\ell-1}F(x)^{n-\ell}\left(1-\frac{F(x-a)}{F(x)}\right)^jf(x)\,\text{d}x=M_j\,,
\]
for $j=1,2$.
Noting also that the choice of the parameter $\beta$ in Theorem \ref{thm:negbin} matches that in Theorem \ref{thm:continuous}, Theorem \ref{thm:negbin} thus gives us that
\begin{multline*}
d_\text{TV}(K_n(a,\ell)-1,Z)\\
\leq\frac{1-(1-\beta)^\ell}{\beta\ell}\mathbb{E}[K_n(a,\ell)-1]\left(\beta
+(1-\beta)\left[(n-\ell-1)\frac{M_2}{M_1}-(n-\ell-2)M_1\right]\right)\,,
\end{multline*}
as required.

\section{Poisson approximation for $K_n$ in the discrete case}\label{sec:poisson}

In this section we establish the Poisson approximation upper bound stated in Theorem \ref{thm:poisson}. Letting $Y\sim\text{Pois}(\lambda)$ as in the statement of the theorem, we use the triangle inequality to write
\begin{equation}\label{eq:triangle}
d_\text{TV}(K_n,Y)\leq d_\text{TV}(K_n,K_n^\star)+d_{TV}(Y,Y+1)+d_\text{TV}(K_n^\star-1,Y)\,.
\end{equation}
Using equation (5) of \cite{daly11}, for the first term on the right-hand side of \eqref{eq:triangle} we have
\begin{equation}\label{eq:triangle2}
d_\text{TV}(K_n,K_n^\star)\leq\frac{\sqrt{\text{Var}(K_n)}}{2\mathbb{E}[K_n]}=\frac{\sqrt{\mathbb{E}[(K_n)_2]-\mathbb{E}[K_n](\mathbb{E}[K_n]-1)}}{2\mathbb{E}[K_n]}\,.
\end{equation}
By Proposition 3 of \cite{daly11} we have
\begin{equation}\label{eq:triangle3}
d_{TV}(Y,Y+1)\leq\frac{1}{2\sqrt{\lambda}}=\sqrt{\frac{\mathbb{E}[K_n]}{4\mathbb{E}[(K_n)_2]}}\,.
\end{equation}
It remains only to bound the final term on the right-hand side of \eqref{eq:triangle}. To that end we note the mixed binomial representation \eqref{eq:MixBinSB} of $K_n^\star-1$ used earlier in the proof of Theorem \ref{thm:discrete}. Recalling that the total variation distance between a binomial $\text{Bin}(n,q)$ distribution and a Poisson distribution of the same mean may be bounded by $q$ (see, for example, equation (1.23) of \cite{barbour92}), a conditioning argument using this mixed binomial representation of $K_n^\star-1$ gives us that
\[
d_\text{TV}(K_n^\star-1,Y^\dagger)\leq\mathbb{E}[q(M)]\,,
\]
where $Y^\dagger$ has the mixed Poisson distribution $Y^\dagger|M\sim\text{Pois}((n-1)q(M))$. Noting that our choice of $\lambda$ is such that $\mathbb{E}[Y^\dagger]=\mathbb{E}[Y]=\lambda$, from Theorem 1.C(ii) of \cite{barbour92} we have that
\[
d_\text{TV}(Y^\dagger,Y)\leq\frac{\text{Var}((n-1)q(M))}{\lambda}=\frac{(n-1)^2\text{Var}(q(M))}{\lambda}\,.
\]
Hence,
\begin{align}
\nonumber d_\text{TV}(K_n^\star-1,Y)&\leq\mathbb{E}[q(M)]+\frac{(n-1)^2\text{Var}(q(M))}{\lambda}\\
\label{eq:triangle4}&=\frac{(n-1)\mathbb{E}[(K_n)_3]}{(n-2)\mathbb{E}[(K_n)_2]}-\frac{(n-2)\mathbb{E}[(K_n)_2]}{(n-1)\mathbb{E}[K_n]}\,,
\end{align}
where the final equality follows from the expressions \eqref{eq:QMoments} for the first two moments of $q(M)$. The conclusion of Theorem \ref{thm:poisson} then follows by combining \eqref{eq:triangle}--\eqref{eq:triangle4}.

\section{Concluding remarks}\label{sec:conc}

In this note we have established explicit error bounds in the logarithmic and Poisson approximation of the number of maxima in a sample of independent, discrete data. The setting in which the underlying data are geometrically distributed was our main example here. Similarly, in the case of continuous data, illustrated by the Gumbel and uniform examples, we establish error bounds in negative binomial approximation for the number of data points within a given threshold of one of the order statistics of our sample. 

We conclude by noting scope for further work. We expect that, with a more sophisticated coupling than used here, error bounds and rates of convergence could be strengthened; for example we would expect to be able to obtain an error bound which decays to zero for a fixed $a$ in the setting of Example \ref{eg:gumbel}. This may require allowing both parameters of our approximating negative binomial distribution $Z$ in Theorem \ref{thm:negbin} to depend on $Q$, rather than only the second parameter. In other cases we similarly note that there is scope to find informative error bounds for a greater range of underlying parameter values. For example, in the geometric setting of Example \ref{eg:geom1} it remains an open problem to find reasonable error bounds for larger values of the parameter $p$. Our geometric and Gumbel examples also illustrate the scope to improve the tightness of our upper bounds.

It would be of interest to extend our findings beyond the case of independent data. We may then no longer have explicit expressions for the distribution of $K_n$ or the order statistics of our data, and so we would expect proofs to become more complex, but note that an important general advantage Stein's method over other techniques to derive such error bounds is its ability to cope with relaxation of independence assumptions. 

Finally, we note that the development of Stein's method for approximation by a logarithmic distribution was an important step in our work. In future work we hope to be able to make use of these tools in applications beyond the setting of the present paper. 

\subsubsection*{Acknowledgements}
The author thanks two anonymous reviewers and an Editor for their careful readings of an earlier version of this manuscript, and constructive comments and suggestions that improved the paper.


\begin{thebibliography}{99}

\bibitem{arratia19} R.~Arratia, L.~Goldstein and F.~Kochman (2019). Size bias for one and all. \emph{Probab. Surv.} {\bf 16}: 1--61. 

\bibitem{barbour92} A.~D.~Barbour, L.~Holst and S.~Janson (1992). \emph{Poisson Approximation}. Oxford University Press, Oxford.

\bibitem{brands94} J.~J.~A.~M.~Brands, F.~W.~Steutel and R.~J.~G.~Wilms (1994). On the number of maxima in a discrete sample. \emph{Statist. Probab. Lett.} {\bf 20}(3): 209--217.

\bibitem{brown99} T.~C.~Brown and M.~J.~Phillips (1999). Negative binomial approximation with Stein's method. \emph{Methodol. Comput. Appl. Probab.} {\bf 1}(4): 407--421.

\bibitem{bruss03} F.~T.~Bruss and R.~Gr\"ubel (2003). On the multiplicity of the maximum in a discrete random sample. \emph{Ann. Appl. Probab.} {\bf 13}(4): 1252--1263.

\bibitem{daly11} F.~Daly (2011). On Stein’s method, smoothing estimates in total variation distance and mixture distributions. \emph{J. Statist. Plann. Inference}  {\bf 141}(7): 2228--2237.

\bibitem{eisenberg09} B.~Eisenberg (2009). The number of players tied for the record. \emph{Statist. Probab. Lett.} {\bf 79}(3): 283--288.

\bibitem{kirschenhofer96} P.~Kirschenhofer and H.~Prodinger (1996). The number of winners in a discrete geometrically distributed sample. \emph{Ann. Appl. Probab.} {\bf 6}(2): 687--694.

\bibitem{olofsson99} P.~Olofsson (1999). A Poisson approximation with applications to the number of maxima in a discrete sample. \emph{Statist. Probab. Lett.} {\bf 44}(1): 23--27.

\bibitem{pakes98} A.~G.~Pakes and Y.~Li (1998). Limit laws for the number of near maxima via the Poisson approximation. \emph{Statist. Probab. Lett.} {\bf 40}(4): 395--401.

\bibitem{pakes97} A.~G.~Pakes and F.~W.~Steutel (1997). On the number of records near the maximum. \emph{Austral. J. Statist.} {\bf 32}(2): 179--192.

\bibitem{rade91} L.~R\"ade (1991). Problem E3436. \emph{Amer. Math. Monthly} {\bf 98}(4): 366.

\bibitem{ross11} N.~Ross (2011). Fundamentals of Stein's method. \emph{Probab. Surv.} {\bf 8}: 210--293.

\bibitem{ross13} N.~Ross (2013). Power laws in preferential attachment graphs and Stein's method for the negative binomial distribution. \emph{Adv. in Appl. Probab.} {\bf 45}(3): 876--893.

\bibitem{shaked07} M.~Shaked and J.~G.~Shanthikumar (2007). \emph{Stochastic Orders}. Springer, New York.

\end{thebibliography}
\end{document}